\documentclass[preprint,12pt]{elsarticle}
\usepackage{amsmath,amssymb,amsthm}
\usepackage{xcolor}
\usepackage{graphicx}
\usepackage{geometry}
\usepackage{caption}
\usepackage{hyperref}
\usepackage{float}
\usepackage{booktabs}
\usepackage{array}

\newtheorem{theorem}{Theorem}[section]
\newtheorem{lemma}[theorem]{Lemma}

\newtheorem{definition}{Definition}

\journal{...}

\begin{document}

\begin{frontmatter}

\title{On the existence, uniqueness and stability of solutions of SDEs with state-dependent variable exponent}
\author{Mustafa Avci}

\affiliation{organization={Faculty of Science and Technology, Applied Mathematics\\  Athabasca University},
            city={Athabasca},
            postcode={T9S 3A3},
            state={AB},
            country={Canada}}
\ead{mavci@athabascau.ca (primary) & avcixmustafa@gmail.com}

\begin{abstract}
We study a time-inhomogeneous nonlinear SDE with drift and diffusion governed by state-dependent variable exponents. This framework generalizes models like the geometric Brownian motion (GBM) and the constant elasticity of variance (CEV), offering flexibility to capture complex dynamics while posing analytical challenges. Using a fixed-point approach, we prove existence and uniqueness, analyze higher-order moments, derive asymptotic estimates, and assess stability. Finally, we illustrate an application where Poisson’s equation admits a probabilistic representation via a time-homogeneous nonlinear SDE with state-dependent variable exponents.
\end{abstract}

\begin{keyword} Stochastic process; state-dependent variable exponent; higher-order moments; asymptotic estimate; stability; Banach fixed-point theorem; Feynman–Kac formula.
\MSC[2020] 37H30; 60G07; 60G65; 60H15
\end{keyword}

\end{frontmatter}

\section{Introduction}\label{Sec1}

We study the following time-inhomogeneous nonlinear stochastic differential equation with state-dependent variable exponent functions $p(X(t)), q(X(t))$
\begin{equation}\label{eq.1a}
\begin{cases}
\begin{array}{rll}
dX(t) &= \mu(t) X(t)^{p(X(t))}dt + \sigma(t) X(t)^{q(X(t))}dW(t),\,\, t\in [0,T], \\
X_{0}&=x_0, \quad x_0 \in (0,\infty),\tag{{${\mathcal{GM}}$}}
\end{array}
\end{cases}
\end{equation}
where $W(t)$ is a Brownian motion on the complete probability space $(\Omega, \mathcal{F}, \mathbb{P})$; $p(\cdot), q(\cdot)$ are differentiable, variable exponent functions; and $0<\mu, \sigma$ are bounded real-valued deterministic continuous functions. \\

Stochastic differential equations (SDEs) play a central role in modeling systems influenced by randomness, with applications spanning finance, physics, biology, and engineering \cite{karatzasshreve91,mao07,oksendal03}. Classical models such as the geometric Brownian motion (GBM) and the constant elasticity of variance (CEV) model have been extensively studied (see e.g.,  \cite{bs1973,cr76}). However, these models often assume fixed exponents in drift and diffusion terms, limiting their ability to capture complex state-dependent dynamics.\\

The proposed model \eqref{eq.1a}, as a novel approach, introduces state-dependent variable exponents $p(\cdot)$ and $q(\cdot)$ in both drift and diffusion terms. Unlike classical models with constant elasticity, this formulation allows the coefficients to respond nonlinearly to the current state. Such flexibility significantly enhances the expressive power of SDEs, enabling better modeling of phenomena where volatility and growth rates vary with the state. However, this nonlinearity creates analytical challenges: standard techniques for proving existence and uniqueness often rely on Lipschitz conditions, which are harder to verify under variable exponent structures. Our work addresses these challenges by introducing a new class of functions, $\mathcal{S}$, for variable exponents to ensure the coefficients satisfy the necessary regularity properties.\\
The proposed model \eqref{eq.1a} generalizes several classical models. For example:
\begin{itemize}
    \item If $p(x) = q(x) = 1$  and $\mu(t)=\mu,\, \sigma(t)=\sigma$, \eqref{eq.1a} reduces to the GBM. A key feature of hypothesis $(\mathbf{h_2})$ is that $p(x) \to 1$ and $q(x) \to 1$ as $x \to \infty$. This implies that our generalized model behaves like GBM for large $x$, but exhibits different dynamics for small $x$.
    \item If $p(x) = 1$ and $q(x) \geq 1$ (excluding $0 \leq q(x) < 1$) and $\mu(t)=\mu,\, \sigma(t)=\sigma$, the model reduces to the CEV model.
\end{itemize}
Thus, \eqref{eq.1a} provides a unified framework that encompasses GBM and CEV as special cases while introducing richer dynamics through state-dependent exponents.\\

In this paper, building on the research findings of \cite{avci25-1,avci25-2,avci25-3}, we differently study a time-inhomogeneous model, and obtain the existence and uniqueness of solution by exploiting the idea of the contraction mapping principle, also known as the Banach fixed-point theorem. Additionally, we investigate higher-order moments ($\geq$ second), asymptotic estimates, and the stability of the solution. Lastly, we provide an application where a probabilistic representation of the solution of Poisson’s equation is given via the solution of a time-homogeneous nonlinear stochastic differential equation with state-dependent variable exponent.

\section{Preliminaries and auxiliary results} \label{Sec2}
We introduce a new class of functions, $\mathcal{S}$, for variable exponents to ensure that the model’s coefficients satisfy Lipschitz and linear growth conditions.
\begin{itemize}
   \item [($\mathbf{h_1}$)] $h(\cdot): (0, \infty) \to [1,\infty)$ is a differentiable function satisfying
   \begin{align*}
     1\leq h^-:=\inf_{x> 0} h(x)\quad \text{and} \quad \sup_{x> 0} h(x):=h^+ <\infty.
   \end{align*}
   \item [($\mathbf{h_2}$)]$\lim_{x \to \infty}h(x)=1 \quad \text{and} \quad \limsup_{x \to \infty}(h(x)-1)\log(x)<\infty.$
   \item [($\mathbf{h_3}$)] There exist real numbers $\delta, M_0, C_0>0$ and $\alpha>0$ with $h^+<1+\alpha$ such that
   \begin{equation*}
   |h^{\prime}(x)| \leq M_0 \mathbf{1}_{(0,\delta]}(x) + C_0 x^{-(1+\alpha)} \mathbf{1}_{(\delta,\infty)}(x).
   \end{equation*}
    \item [$(\mathbf{D})$] $\mu, \sigma: [0,T] \to (0,\infty)$ are deterministic continuous functions such that
   \begin{align*}
   &0<\mu^-:=\min_{t \in [0,T]} \mu(t),\quad  \max_{t \in [0,T]} \mu(t):=\mu^+ <\infty.\\
   &0<\sigma^-:=\min_{t \in [0,T]} \sigma(t), \quad \max_{t \in [0,T]} \sigma(t):=\sigma^+ <\infty.
   \end{align*}
\end{itemize}
We say that the function $h(\cdot)$ belongs to the class $\mathcal{S}$ if it satisfies the hypotheses $(\mathbf{h_1})-(\mathbf{h_3})$.

\subsection{Illustrative examples}

We present several examples of state-dependent exponent functions that satisfy assumptions $(\mathbf{h_1})-(\mathbf{h_3})$ and therefore fall within the framework of the model \eqref{eq.1a}. These examples illustrate how the proposed formulation extends classical SDE models by allowing the elasticity of the drift and diffusion terms to vary with the state variable.

\paragraph{Example 1}
Let
\[
p(x)=1+\frac{a}{1+x}, \qquad q(x)=1+\frac{b}{1+x}, \qquad a,b>0 .
\]
Then \eqref{eq.1a} becomes
\[
dX(t)=\mu(t)X(t)^{1+\frac{a}{1+X(t)}}dt
      +\sigma(t)X(t)^{1+\frac{b}{1+X(t)}}dW(t).
\]
Since $p(x),q(x)\to 1$ as $x\to\infty$, the dynamics asymptotically approach the GBM, while for small values of $x$ the model exhibits stronger nonlinear drift and diffusion effects. Such a structure can be useful for modeling systems where growth and variability are amplified when the state variable is small.

\paragraph{Example 2}
Let
\[
p(x)=1, \qquad q(x)=\gamma+\frac{c}{1+x},
\]
where $\gamma>0$ and $c>0$. Then \eqref{eq.1a} becomes
\[
dX(t)=\mu(t)X(t)\,dt
      +\sigma(t)X(t)^{\gamma+\frac{c}{1+X(t)}}dW(t).
\]
In this case the diffusion elasticity varies with the state variable but
approaches the constant value $\gamma$ as $x\to\infty$. Hence the model can be
viewed as a state-dependent generalization of the classical CEV process.

\paragraph{Example 3}
Let
\[
p(x)=1+\frac{a}{1+x^{2}}, \qquad q(x)=1+\frac{b}{1+x^{2}}, \qquad a,b>0 .
\]
Then \eqref{eq.1a} becomes
\[
dX(t)=\mu(t)X(t)^{1+\frac{a}{1+X(t)^{2}}}dt
      +\sigma(t)X(t)^{1+\frac{b}{1+X(t)^{2}}}dW(t).
\]
Here the nonlinear effects are most pronounced near the origin and gradually vanish as the state grows. Consequently the model behaves like a GBM for large values of $x$ while allowing richer nonlinear dynamics for moderate and small values of the state variable.\\

Next, we will continue with some basic concepts of the measure-theoretic probability and the theory of stochastic process.\\

Let $\{\mathcal{F}_t\}_{t \geq 0}$  be a \emph{filtration} denoted by $\mathcal{F}_t$. Then $(\Omega, \mathcal{F}, \mathcal{F}_t, \mathbb{P})$ is a \emph{filtered probability space} satisfying the usual conditions (see, e.g., \cite{mao07,oksendal03}).
A \emph{stochastic process} $ X(t):=\{X(t)\}_{0\leq t\leq T} $, defined for $t \geq 0 $, is said to be \emph{progressively measurable} with respect to a filtration $ \mathcal{F}_t $ if, for every $ t \in [0, T]$, the mapping
\begin{equation}
X: [0,t] \times \Omega \to \mathbb{R}, \quad (s, \omega) \mapsto X(s, \omega)
\end{equation}
is measurable with respect to the product $\sigma$-field $\mathcal{B}([0,t]) \otimes \mathcal{F}_t $, where $ \mathcal{B}([0,t]) $ denotes the Borel $\sigma$-field on $[0,t] $. Note that since the stochastic process $ X(t)$ given by \eqref{eq.1a} is adapted to $\mathcal{F}_t$ and has continuous paths $\mathbb{P}-$almost surely ($\mathbb{P}-$a.s.), it is progressively measurable.
In order to show the existence of a solution, we are going to employ the Banach fixed point theorem in the space
\begin{equation}
\mathcal{M_{T}}=\left\{X(t): X \text{ is progressively measurable, } \int_{\Omega}\int_{0}^{T}|X(t,\omega)|^{2}dt d\mathbb{P}(\omega)<\infty   \right\}
\end{equation}
with the norm
\begin{equation}
\|X\|^{2}_{\mathcal{T}}=\mathbb{E}\left[\int_{0}^{T}|X(t)|^{2}dt\right].
\end{equation}
The normed space $(\mathcal{M_{T}}, \|\cdot\|_{\mathcal{T}})$ is complete since it is a closed subspace of the space $L^{2}([0,T] \times \Omega)$ of functions of two variables $(t, \omega) \in [0,T] \times \Omega$ that are square-integrable with respect to the product measure $d\left(\nu \times \mathbb{P} \right)(t, \omega)=dt d\mathbb{P}(\omega)$, where $\nu$ is Lebesgue measure on $[0,T]$.

\begin{lemma}\label{Lem:2.1}
Assume $X(t) \in \mathcal{M_T}$ and $h(\cdot)$ satisfies $(\mathbf{h_1})$ and $(\mathbf{h_2})$. Then $X^{h}(t):= X(t)^{h(X(t))} \in \mathcal{M_T}$.
\end{lemma}
\begin{proof} It is enough to apply ($\mathbf{h_2}$). If $0<x\leq 1$, there exists $0<\epsilon\leq 1$ such that
\begin{equation}\label{eq.1c}
x^{h(x)}\leq \epsilon^{h^-}\leq \epsilon^{h^-}(1+x).
\end{equation}
If $1<x<\infty$, there exist real numbers $M_{\infty}, R_{\infty}>0$ such that
\begin{equation}\label{eq.1d}
(h(x)-1)\log(x) \leq M_{\infty}, \quad \forall x >R_{\infty},
\end{equation}
which yields
\begin{equation}\label{eq.1e}
e^{(h(x)-1)\log (x)}\leq e^{M_{\infty}} \Rightarrow x^{h(x)}\leq e^{M_{\infty}}(1+x).
\end{equation}
Therefore, for any $x \in (0,\infty)$
\begin{equation}\label{eq.1e}
x^{h(x)}\leq K(1+x),
\end{equation}
where $K:= \max\{\epsilon^{h^-}, e^{M_{\infty}}\}$. Hence, considering all these along with the fact that $X(t) \in \mathcal{M_T}$,  it follows that
\begin{equation}\label{eq.1g}
\mathbb{E}\left[\int_{0}^{T}X^{2h}(t)dt\right]\leq 2K^2(T+1)\mathbb{E}\left[\int_{0}^{T}X^2(t)ds \right]=2K^2(T+1)\|X\|^{2}_{\mathcal{T}}<\infty.
\end{equation}
\end{proof}

\begin{lemma}\label{Lem:2.2}
Assume that $X_0=x_0 > 0$, and $p(\cdot), q(\cdot)$ satisfy $(\mathbf{h_1})$ and $(\mathbf{h_3})$. Then $X(t)$ remains strictly positive in the state space $(0,\infty)$ $\mathbb{P}-$a.s. for all $t \in [0,T]$. In particular, the boundary $x=0$ is natural; that is, starting from $x_0$, it is never hit in finite time.
\end{lemma}
\begin{proof}
We use the local (differential) form of Feller’s non-attainability test \cite{elt09} to analyze the behavior of the processes $X(t)$ at the boundary of its state space $(0,\infty)$.
For a general SDE
$$
dX(t) = \mu\left(t, X(t)\right)dt + \sigma\left(t, X(t)\right)dW(t)
$$
if
\begin{equation}\label{eq.2}
\lim_{x \to 0^+}\left(\mu(t,x)-\frac{1}{2}\frac{\partial \sigma^2}{\partial x}(t,x)\right)\geq 0
\end{equation}
holds, then the boundary $X(t)=0$ is non-attainable for the process $X(t)$ given that $x_0>0$. Now, we apply \eqref{eq.2} to \eqref{eq.1a}.
Then, for $0<x\leq\delta<1$,
\begin{align}\label{eq.2a}
\mathcal{T}(t,x)=\mu(t) x^{p(x)}-\frac{\sigma^2(t)}{2}\frac{\partial  x^{2q(x)}}{\partial x}\geq \mu^- x^{p(x)}-(\sigma^+)^2 \left(M_0x^{2q(x)}|\log(x)|+x^{2q(x)-1}q^+\right).
\end{align}
Using the assumptions $(\mathbf{h_1})$ and $(\mathbf{h_3})$, it follows
\begin{align}\label{eq.2b}
\lim_{x \to 0^+}\mathcal{T}(t,x)& \geq \lim_{x \to 0^+} \left(\mu^- x^{p(x)}-(\sigma^+)^2 \left(M_0x^{2q(x)}|\log(x)|+x^{2q(x)-1}q^+\right)\right)= 0.
\end{align}
Hence $X(t)>0$ $\mathbb{P}-$a.s. for all $t \in [0,T]$, which means \eqref{eq.1a} is well-defined.
\end{proof}

\begin{lemma}\label{Lem:2.3}\cite{mao07}
Let $m\geq 2$ be a real number. Let $f \in \mathcal{M_{T}}$ such that
\begin{align}\label{eq.2ba}
\mathbb{E}\left[\int_{0}^{T}|f(s)|^{m} ds\right]<\infty.
\end{align}
Then
\begin{align}\label{eq.2ba}
\mathbb{E}\left[\left|\int_{0}^{T}f(s) dW(s)\right|^{m}\right]\leq \left(\frac{m(m-1)}{2}\right)^{\frac{m}{2}}\,T^{\frac{m-2}{2}}\,\mathbb{E}\left[\int_{0}^{T}|f(s)|^{m} ds\right].
\end{align}
In particular, for $m=2$, there is equality, i.e. \eqref{eq.2ba} gives It\^{o} isometry.
\end{lemma}

\begin{lemma}\label{Lem:2.4}\cite{mao07}
Let $\alpha, \beta, T$ be any positive numbers. Let $g \in \mathcal{L}_{2}([0,T];\mathbb{R})$ (i.e. the family of $\mathbb{R}$-valued $\mathcal{F}_{t}$-adapted processes) such that
\begin{align}\label{eq.2bc}
\mathbb{P}\left\{\int_{t_0}^{T}|g(t)|^{2}dt<\infty \right\}=1 \quad \text{for all} \quad 0\leq t_0 < T.
\end{align}
Then
\begin{align} \label{eq.2bd}
\mathbb{P}\left\{\sup_{t\leq T} \left[\int_{t_0}^t g(s)dW(s)-\frac{\alpha}{2} \int_{t_0}^t |g(s)|^2 ds\right]>\beta \right\}\leq e^{-\alpha \beta}.
\end{align}
\end{lemma}

\section{The main result} \label{Sec3}
\begin{theorem}\label{Thrm:3.1}
Assume that $p(\cdot), q(\cdot) \in \mathcal{S}$ and condition $(\mathbf{D})$ holds. Also, suppose the initial condition $X_0=X(0)=x_0$ is $\mathcal{F}_0$-measurable, independent of $W(t)$, and has a finite second moment. Then there exists a unique global solution $X(t)$ of \eqref{eq.1a} with continuous paths satisfying $\mathbb{E}\left[\int_{0}^{T}|X(t)|^{2}dt\right]<\infty$.
\end{theorem}
\begin{proof}
Let us fix $T>0$. Let $\Phi$ be the function that maps a process $X(t)$ into a process $\Phi X(t)$ defined by
\begin{equation} \label{eq.1b}
\Phi X(t) = x_0 + \int_0^t  \mu(s) X^{p}(s)ds + \int_0^t \sigma(s) X^{q}(s)dW(s).
\end{equation}
In this context, the solution of \eqref{eq.1a} is a process $X(t)$ that satisfies the equation
\begin{equation}
\Phi X(t)=X(t),
\end{equation}
that is, it is a fixed point of $\Phi$. We shall present the proof in four steps.\\
\textbf{Step 1.} Let $X(t) \in \mathcal{M_{T}}$. Then
\begin{equation} \label{eq.4}
(\Phi X(t))^2 \leq 3x^2_0+ 3(\mu^+)^2\left(\int_0^t  X^{p}(s) ds\right)^2 + 3(\sigma^+)^2\left(\int_0^t X^{q}(s)dW(s)\right)^2.
\end{equation}
Taking the expectation and integrating with respect to $t$ we obtain
\begin{align} \label{eq.5}
\int_{0}^{T}\mathbb{E}\left[(\Phi X(s))^2\right]dt &\leq 3T\mathbb{E}[x^2_0]+ 3(\mu^+)^2\int_{0}^{T}\mathbb{E}\left[\left(\int_0^t  X^{p}(s) ds\right)^2\right]dt \nonumber \\
&+ 3(\sigma^+)^2\int_{0}^{T}\mathbb{E}\left[\left(\int_0^t X^{q}(s)dW(s)\right)^2\right]dt.
\end{align}
Now, applying the Cauchy-Schwarz inequality to the second term, the It\^o isometry to the third term (since $X^{q}(t)\in \mathcal{M_T}$) in \eqref{eq.5}, and using Lemma \ref{Lem:2.1}, we obtain
\begin{align} \label{eq.6}
\int_{0}^{T}\mathbb{E}\left[(\Phi X(s))^2\right]dt &\leq 3T\mathbb{E}[x^2_0]+ 3[(\mu^+)^2+(\sigma^+)^2] K^2\|X\|^{2}_{\mathcal{T}}\int_{0}^{T} 2t(t+1)dt<\infty.
\end{align}
Since $(\Phi X(s))^2$ is measurable on $[0,T] \times \Omega$, by the Fubini-Tonelli theorem, one concludes that $\Phi : \mathcal{M_{T}} \to \mathcal{M_{T}}$.
Note that if one follows similar arguments as with the proof of Lemma 3.1\,$(i)$ of \cite{avci25-2} and consider $(\mathbf{D})$, it is straightforward to show that there exists a constant $L=L(\mu_+,\sigma_+)>0$ such that $\forall x,y\in (0,\infty), t \in [0,T]$, it holds:
\begin{equation}\label{eq.7}
   |\mu(t)(x^{p(x)}-y^{p(y)})| +|\sigma(t)(x^{q(x)}-y^{q(y)})| \leq L|x-y|,
\end{equation}
where  $\mu_{+}=\max\{\mu^-, \mu^+ \}$ and $\sigma_{+}=\max\{\sigma^-, \sigma^+ \}$. Let $X(t),Y(t) \in \mathcal{M_{T}}$. Then applying the Cauchy-Schwarz inequality and the It\^o isometry alon with \eqref{eq.7} gives
\begin{align} \label{eq.8}
\|\Phi X-\Phi Y\|^{2}_{\mathcal{T}}& \leq 2\int_{0}^{T}\left(\mu_{+}^2 \mathbb{E}\left[\left(\int_{0}^{t} \Delta(s) ds\right)^{2}\right]+ \sigma_{+}^2 \mathbb{E}\left[\left(\int_{0}^{t} \Delta(s) dW(s)\right)^{2}\right]\right)dt \nonumber \\
& \leq c^2 \|X-Y\|^{2}_{\mathcal{T}},
\end{align}
where $\Delta(t):=X(t)-Y(t)$ and $c=\sqrt{2L^2T(T\mu_{+}^2+\sigma_{+}^2)}$. Finally, if we choose $0<T_{*}\leq T$ so that $c <1$, it turns out that $\Phi$ is a contraction from $\mathcal{M_{T_*}}$ into $\mathcal{M_{T_*}}$, i.e.
\begin{align} \label{eq.9}
\|\Phi X-\Phi Y\|_{\mathcal{T_*}}& \leq c \|X-Y\|_{\mathcal{T_*}}.
\end{align}
\textbf{Step 2.}
Let $X_{n}(t)$ be a sequence in $\mathcal{M_{T_*}}$. Define the sequence
$$X_{n+1}(t)=\Phi X_{n}(t) \quad \text{with} \quad X_{1}(t)=\Phi X_{0}(t).$$
Then consecutive approximations and induction lead to
\begin{align} \label{eq.11}
\|X_{n+1}-X_{n}\|_{\mathcal{T_*}}\leq c^n\|X_{1}-X_{0}\|_{\mathcal{T_*}}.
\end{align}
Thus, for $n<m$
\begin{align} \label{eq.12}
\|X_{m}-X_{n}\|_{\mathcal{T_*}}\leq (c^{m-1}+c^{m-2}+\cdots +c^n)\|X_{1}-X_{0}\|_{\mathcal{T_*}}.
\end{align}
Since $c<1$, the sum of geometric series $\sum_{k=n}^{m-1}c^k$ tends to $0$ as $n,m \to \infty$. Thus, $X_{n}$ is a Cauchy sequence in $\mathcal{M_{T_*}}$. Since $\mathcal{M_{T_*}}$ is complete and $\Phi$ is continuous, passing to the limit in $X_{n+1}(t)=\Phi X_{n}(t)$ gives $X(t)=\Phi X(t)$. Thus, $X(t)$ is a solution of \eqref{eq.1a} in $\mathcal{M_{T_*}}$. This solution is unique. Argue otherwise and assume that $\Phi$ has another fixed point, $Y(t) \in \mathcal{M_{T_*}}$ such that $X(t)\neq Y(t)$ $\mathbb{P}-$a.s. Then
\begin{align} \label{eq.14}
\|\Phi X-\Phi Y\|_{\mathcal{T_*}}& < \|X-Y\|_{\mathcal{T_*}}=\|\Phi X-\Phi Y\|_{\mathcal{T_*}},
\end{align}
which is a contradiction.\\
\textbf{Step 3.}
We already solved the equation on $[0,T_*]$ with initial condition $X(0)$. Now we want to ensure that the solution can be continued to $[T_*,2T_*]$, then $[2T_*,3T_*]$, and so on, until we cover $[0,T]$.
In doing so, by applying the same arguments as before, we obtain for $t \leq T_*$
\begin{align} \label{eq.15}
\mathbb{E}\left[X^2(t)\right] &\leq 3\mathbb{E}[x^2_0]+ 3[2(T+1)K^2(\mu^2_+ + \sigma^2_+)]\int_{0}^{t} \mathbb{E}[X^2(s)]ds.
\end{align}
Applying the Gronwall lemma shows that $\mathbb{E}\left[X^2(T_*)\right]$ is finite; that is, at $t=T_*$, $X(T_*)$ doesn't explode. Now, if we use $X(T_*)$ as the initial condition for the next interval $[T_*,2T_*]$ and repeat the argument for each interval until $nT_*\geq T$, we obtain the global existence of the solution on $[0,T]$, i.e. the solution is in $\mathcal{M_{T}}$. \\
\textbf{Step 4.}
Lastly, we will show that a solution of \eqref{eq.1a} always stays in the class $\mathcal{M_{T}}$   for all $t \in [0,T]$. For integer $n\geq 1$, define the stopping time
\begin{equation}\label{eq.16}
\tau_n=T \wedge \inf\{ t \geq 0: |X(t)|\geq n\}.
\end{equation}
Then, $\tau_n \uparrow T$ $\mathbb{P}-$a.s. Set $X_n(t)=X(t \wedge \tau_n)$ for $t \in [0,T]$. Then following the same arguments as before provides
\begin{align} \label{eq.17}
\mathbb{E}[X^{2}_n(t)]& \leq 3\mathbb{E}[x_0^2]+ 3\mu^2_+\mathbb{E}\left[\left(\int_0^{t \wedge \tau_n} X_n^{p}(s)ds\right)^2\right] +3\sigma^2_+\mathbb{E}\left[\left(\int_0^{t \wedge \tau_n} X_n^{q}(s)dW(s)\right)^2\right]\nonumber \\
& \leq 3\mathbb{E}[x_0^2] + 6 K^2(T+1)(\mu^2_+ +\sigma^2_+)\mathbb{E}\left[\int_0^{t \wedge \tau_n} X^{2}_n(s)ds \right] \nonumber \\
& \leq c_1 + c_2\int_{0}^{t} \mathbb{E}[X^{2}_n(s)]ds
\end{align}
where $c_1=3\mathbb{E}[x_0^2]$ and $c_2=3[2(T+1)K^2(\mu^2_+ + \sigma^2_+)]$ are independent of $n$. If we define $\varphi_n(t):=\mathbb{E}[X^{2}_n(t)]$, and apply the Gronwall lemma to \eqref{eq.17} we obtain
\begin{align} \label{eq.18}
\varphi_n(T) \leq c_1e^{c_2T}= c_3  \Longrightarrow \sup_{n}\varphi_n(T)<\infty \,\, \text{as } n \to \infty.
\end{align}
Since $X_n(t) \to X(t)$  $\mathbb{P}-$a.s. as  $n \to \infty$, by Fatou's lemma it reads
\begin{align} \label{eq.19}
\mathbb{E}[X^{2}(t)]=\mathbb{E}[\liminf X^{2}_n(t)]\leq \liminf\mathbb{E}[X^{2}_n(t)]\leq \sup_{n}\mathbb{E}[X^{2}_n(t)],
\end{align}
which implies that $\mathbb{E}[X^{2}(t)]$ is bounded and belongs to Lebesgue space $L([0,T])$. Therefore, the solution  $X(t)$ of \eqref{eq.1a} remains in the class $\mathcal{M_{T}}$ for any $T$.
\end{proof}

\section{The properties of the solution}\label{Sec4}

In this section, by assuming that $X(t)$, with the initial value $X(t_0)=x_0$, $0\leq t_0\leq T$, where $x_0$ is a $\mathcal{F}_{t_0}$-measurable random variable, is the unique solution for the equation \eqref{eq.1a}, we establish some important properties of the solution.

\subsection{The higher-order moments of the solution}\label{Sec4.1}
\begin{theorem}\label{Thrm:4.1}
Let $m\geq 2$ be a real number and $\mathbb{E}\left[x_0^{m}\right]<\infty$. Then
\begin{align} \label{eq.20}
\mathbb{E}\left[X^{m}(t)\right]& \leq \left(3^{m-1}\mathbb{E}[x_0^2] +(t-t_0)(A+B)\right)\,e^{t(A+B)},
\end{align}
where $A=6^{m-1}(\mu^{+})^{m}(t-t_0)^{m-1} K^{m}$ and $B=6^{m-1}(\sigma^{+})^{m}\left(\frac{m(m-1)}{2}\right)^{\frac{m}{2}}\,(t-t_0)^{\frac{m-2}{2}}\,K^{m}$.
\end{theorem}
\begin{proof}
Let's first define the stopping time $\hat{\tau}_n$ by
\begin{equation} \label{eq.20d}
\hat{\tau}_n=T \wedge \inf\{t\geq0;\, X_{n}(t) \notin (0,n),\,\, n \in \mathbb{N}\},
\end{equation}
from which we have $\hat{\tau}_n \uparrow T$ $\mathbb{P}-$a.s. Set $X_n(t)=X(t \wedge \hat{\tau}_n)$ for $t \in [0,T]$. Then proceeding as before but this time using Lemma \ref{Lem:2.3} and the linear growth \eqref{eq.1e} provides
\begin{align} \label{eq.20e}
\mathbb{E}\left[X^{m}_n(t)\right]& \leq 3^{m-1}\mathbb{E}[x_0^2]+ 3^{m-1}(\mu^{+})^{m}\mathbb{E}\left[\left(\int_{t_0}^{t \wedge \hat{\tau}_n} X_n^{p}(s)ds\right)^m\right]\nonumber \\
& +3^{m-1}(\sigma^{+})^{m}\mathbb{E}\left[\left(\int_{t_0}^{t \wedge \hat{\tau}_n} X_n^{q}(s)dW(s)\right)^m\right]\nonumber \\
&\leq 3^{m-1}\mathbb{E}[x_0^2]+ 3^{m-1}(\mu^{+})^{m}(t-t_0)^{m-1}\mathbb{E}\left[\int_{t_0}^{t \wedge \hat{\tau}_n} X^{mp}_n(s)ds \right]\nonumber \\
&+3^{m-1}(\sigma^{+})^{m}\left(\frac{m(m-1)}{2}\right)^{\frac{m}{2}}\,(t-t_0)^{\frac{m-2}{2}}\,\mathbb{E}\left[\int_{t_0}^{t \wedge \hat{\tau}_n} X^{qm}_n(s)ds\right]  \nonumber \\
&\leq 3^{m-1}\mathbb{E}[x_0^2] +(t-t_0)(A+B)+ (A+B)\int_{t_0}^{t \wedge \hat{\tau}_n} \mathbb{E}\left[X^{m}_n(s)\right]ds \nonumber \\
&\leq 3^{m-1}\mathbb{E}[x_0^2] +(t-t_0)(A+B)+ (A+B)\int_{t_0}^{t} \mathbb{E}\left[X^{m}(s \wedge \hat{\tau}_n)\right]ds.
\end{align}
Applying the Gronwall inequality gives
\begin{align} \label{eq.20f}
\mathbb{E}\left[X^{m}(t \wedge \hat{\tau}_n)\right]& \leq \left(3^{m-1}\mathbb{E}[x_0^2] +(t-t_0)(A+B)\right)\,e^{t(A+B)}.
\end{align}
Finally, letting $n \to \infty$, and considering that $\hat{\tau}_n \to T$ as $n \to \infty$ $\mathbb{P}-$a.s., gives
\begin{align} \label{eq.20g}
\mathbb{E}\left[X^{m}(t)\right]& \leq \left(3^{m-1}\mathbb{E}[x_0^2] +(t-t_0)(A+B)\right)\,e^{t(A+B)}.
\end{align}
\end{proof}

\subsection{The asymptotic estimate of the solution}\label{Sec4.2}
\begin{theorem}
Assume that the linear growth condition \eqref{eq.1e} holds. Then the solution $X(t)$ of \eqref{eq.1a} $\mathbb{P}$-a.s. has the following property
\begin{align} \label{eq.20h}
\limsup_{t \to \infty}\frac{\log(X(t))}{t}\leq \hat{K},
\end{align}
where $\hat{K}=4\max\{K^2(\sigma^{+})^2, K\mu^+ \}$.
\end{theorem}

\begin{proof}
Applying It\^{o}'s formula gives
\begin{align} \label{eq.20k}
d(\log(1+X^2(t)))&=\frac{1}{1+X^2(t)}\left[2 X(t) \mu(t) X^{p}(t) + \sigma^2(t) X^{2q}(t) \right]dt \nonumber \\
&-\frac{2 X^2(t) \sigma^2(t) X^{2q}(t)}{(1+X^2(t))^2}dt+\frac{2X(t) \sigma(t) X^{q}(t)}{1+X^2(t)}dW(t).
\end{align}
Note that from the linear growth \eqref{eq.1e}, one can obtain
\begin{align} \label{eq.20m}
x\mu(t) x^{p(x)} + \frac{1}{2}\sigma^2(t) x^{2q(x)}\leq \hat{K}(1+x^2).
\end{align}
Then
\begin{align} \label{eq.20n}
\log(1+X^2(t))&=\log(1+x^2_0)+2\hat{K}(t-t_0) \nonumber \\
&-2\int_{t_0}^{t}\frac{(\sigma(s) X(s)  X^{q}(s))^2}{(1+X^2(s))^2}ds+2\int_{t_0}^{t}\frac{ \sigma(s) X(s) X^{q}(s)}{1+X^2(s)}dW(s).
\end{align}
Then, for any integer $n\geq t_0$, using Lemma \ref{Lem:2.4} gives
\begin{align} \label{eq.20p}
\mathbb{P}\left\{\sup_{t_0\leq t \leq n} \left[2\int_{t_0}^{t}\frac{ \sigma(s) X(s) X^{q}(s)}{1+X^2(s)}dW(s)-2\int_{t_0}^{t}\frac{(\sigma(s) X(s)  X^{q}(s))^2}{(1+X^2(s))^2}ds\right]>2\log(n) \right\}\leq \frac{1}{n^2}.
\end{align}
Define the events
\begin{align} \label{eq.20r}
E_{n}=\left\{\sup_{t_0\leq t \leq n} \left[\int_{t_0}^{t}\frac{ \sigma(s) X(s) X^{q}(s)}{1+X^2(s)}dW(s)-\int_{t_0}^{t}\frac{(\sigma(s) X(s)  X^{q}(s))^2}{(1+X^2(s))^2}ds\right]>\log(n) \right\}.
\end{align}
Then
\begin{align} \label{eq.20s}
\sum_{n=1}^{\infty}\mathbb{P}(E_{n})\leq \sum_{n=1}^{\infty}\frac{1}{n^2}<\infty,
\end{align}
which, by the Borel-Cantelli lemma, implies
\begin{align} \label{eq.20st}
\mathbb{P}(\limsup_{n \to \infty}E_{n})=\mathbb{P}(E_{n}\, \text{i.o})=0.
\end{align}
Therefore, for almost all $\omega \in \Omega$ there exists a random integer $n_0=n_0(\omega)\geq t_0+1$ such that
\begin{align} \label{eq.20t}
\sup_{t_0\leq t \leq n} \left[\int_{t_0}^{t}\frac{ \sigma(s) X(s) X^{q}(s)}{1+X^2(s)}dW(s)-\int_{t_0}^{t}\frac{(\sigma(s) X(s)  X^{q}(s))^2}{(1+X^2(s))^2}ds\right]\leq\log(n),\quad \text{if }\,n\geq n_0.
\end{align}
Then for all $t_0\leq t\leq n$, $n\geq n_0$ $\mathbb{P}-a.s.$ we have
\begin{align} \label{eq.20u}
\int_{t_0}^{t}\frac{ \sigma(s) X(s) X^{q}(s)}{1+X^2(s)}dW(s)\leq \int_{t_0}^{t}\frac{(\sigma(s) X(s)  X^{q}(s))^2}{(1+X^2(s))^2}ds+\log(n).
\end{align}
Thus using \eqref{eq.20u} in \eqref{eq.20n} gives
\begin{align} \label{eq.20v}
\log(1+X^2(t))&=\log(1+x^2_0)+2\hat{K}(t-t_0)+2\log(n).
\end{align}
Hence, for almost all $\omega \in \Omega$, when $n\geq n_0$ and $n-1\leq t\leq n$, we can write
\begin{align} \label{eq.20y}
\frac{1}{t}\log(1+X^2(t))&=\frac{1}{n-1}\left(\log(1+x^2_0)+2\hat{K}(n-t_0)+2\log(n)\right),
\end{align}
from which one can, $\mathbb{P}$-a.s., obtain
\begin{align} \label{eq.20z}
\limsup_{t \to \infty}\frac{\log(X(t))}{t}&\leq \limsup_{t \to \infty}\frac{1}{2t}\log(1+X^2(t))\nonumber \\
&\leq \limsup_{n \to \infty}\frac{1}{2(n-1)}\left(\log(1+x^2_0)+2\hat{K}(n-t_0)+2\log(n)\right)=\hat{K}.
\end{align}
This completes the proof.
\end{proof}

\subsection{The stability of the solution}\label{Sec4.3}

\begin{definition} \label{Def:4.1}
A solution $X_{\xi}(t)$ of \eqref{eq.1a} with initial value $X(t_0)=\xi$, $0\leq t_0\leq T$, is said to be stable in the $m$th power if for all $\varepsilon > 0$ there exists $\delta > 0$ such that when \newline $\mathbb{E}\left[|\xi - \eta|^m\right]<\delta$, then
\begin{align} \label{eq.21}
\mathbb{E}\left[\sup_{0\leq s \leq t}|X_{\xi}(s)-X_{\eta}(s)|^m\right] & \leq \varepsilon,
\end{align}
where $X_{\eta}(t)$ is another solution of \eqref{eq.1a} with initial value $X(t_0)=\eta$, and $\xi,\eta$ are $\mathcal{F}_{t_0}$-measurable random variables.
\end{definition}

\begin{theorem}
Let $X_{\xi}(t)$ and $X_{\eta}(t)$ be two solutions of \eqref{eq.1a} starting from positions $\xi$ and $\eta$, respectively, where $\xi$ and $\eta$ are $\mathcal{F}_{t_0}$-measurable random variables. Then the solution of \eqref{eq.1a} is stable in the $m$th power.
\end{theorem}

\begin{proof}
We fix $t \in [0,T]$, and consider the sequence $x_n \to x$ $\mathbb{P}-$a.s. as $n \to \infty $ with $x \in (0,\infty)$, and $(x_n) \subset (0,\infty)$. Let $X_{x_n}(t)$ and $X_{x}(t)$ denote the solutions of \eqref{eq.1a} starting from positions $x_n$ and $x$, respectively. Then for any integer $m\geq 2$, it reads
\begin{align} \label{eq.22}
\sup_{0\leq s \leq t}|X_{x_n}(s)-X_{x}(s)|^{m} & \leq 3^{m-1} |x_n-x|^m \nonumber \\
& + 3^{m-1}(\mu^{+})^m\sup_{0\leq s \leq t}\left|\int_0^s \left(X_{x_n}(r)^{p(X_{x_n}(r))}-X_{x}(r)^{p(X_{x}(r))}\right) dr\right|^{m} \nonumber \\
& + 3^{m-1} (\sigma^{+})^m \sup_{0\leq s \leq t}\left|\int_0^s \left(X_{x_n}(r)^{q(X_{x_n}(r))}-X_{x}(r)^{q(X_{x}(r))}\right) dW(r)\right|^{m}.
\end{align}
Applying the Burkholder-Davis-Gundy inequality \cite{bg1970, davis1970} and the H\"{o}lder inequality  together yields
\begin{align} \label{eq.23}
\mathbb{E}\left[\sup_{0\leq s \leq t}|X_{x_n}(s)-X_{x}(s)|^m\right] & \leq 3^{m-1} \mathbb{E}\left[|x_n-x|^m\right] \nonumber \\
& + 3^{m-1}(\mu^{+})^m L^m t^{m-1} \int_0^t \mathbb{E}\left[\sup_{0\leq s \leq r}\left|X_{x_n}(s)-X_{x}(s)\right|^m\right] dr \nonumber \\
& + 3^{m-1} C_m\sigma^{+})^m L^m t^{(m-2)/2} \int_0^t \mathbb{E}\left[\sup_{0\leq s \leq r}\left|X_{x_n}(s)-X_{x}(s)\right|^m\right] dr \nonumber \\
& \leq 3^{m-1} \mathbb{E}\left[|x_n-x|^m\right] + \bar{L}\int_0^t \mathbb{E}\left[\sup_{0\leq s \leq r}\left|X_{x_n}(s)-X_{x}(s)\right|^m\right] dr.
\end{align}
where $\bar{L}=3^{m-1}L^m((\mu^{+})^m t^{m-1}+ C_m(\sigma^{+})^m t^{(m-2)/2})$. Lastly, applying the Gronwall inequality provides
\begin{align} \label{eq.24}
\mathbb{E}\left[\sup_{0\leq s \leq t}|X_{x_n}(s)-X_{x}(s)|^m\right] & \leq 3^{m-1}\, e^{T\bar{L}}\, \mathbb{E}\left[|x_n-x|^m\right].
\end{align}
Considering that $x_n \to x$ $\mathbb{P}-$a.s., there exists a number $\delta^{1/m}>0$ and a positive integer $n_0$ such that whenever $n\geq n_0$, it holds $|x_n-x|^m< \delta$, and hence, $\mathbb{E}\left[|x_n-x|^m\right]<\delta$. Thus, if we let $\delta=\frac{\varepsilon}{ 3^{m-1}\, e^{T\bar{L}}}$ in \eqref{eq.24}, it follows
\begin{align} \label{eq.24a}
\mathbb{E}\left[\sup_{0\leq s \leq t}|X_{x_n}(s)-X_{x}(s)|^m\right] & \leq \varepsilon,
\end{align}
which, by the Dominated Convergence theorem, implies
\begin{align} \label{eq.25}
\sup_{0\leq s \leq t}|X_{x_n}(s)-X_{x}(s)|\to 0\, \text{ in }\, L^{m}([0,T])\, \text{ whenever }\, x_n \to x\,\,\, \mathbb{P}-a.s.
\end{align}
The proof is complete.
\end{proof}

\section{Application}\label{Sec5}
SDEs have numerous applications, one of the most notable being their role in representing solutions to (deterministic) differential equations through the Feynman–Kac formula. This formula provides a fundamental link between SDEs and differential equations, offering a probabilistic framework for analyzing and solving such equations.\\

Let us consider Poisson’s equation
\begin{equation}\label{eq.1ab}
\begin{cases}
\begin{array}{rll}
Lu(x)-cu(x)& =-f(x)\quad \text{in } \mathcal{O}=(a,b)\subset (0,\infty),  \\
u(x)&=0 \quad \text{on } \partial \mathcal{O},
\end{array}
\end{cases}
\end{equation}
where $f$ is a real-valued $C^1$ function; $\mu, \sigma>0$ and $c \geq 0$ real parameters; $L$ is the linear differential operator
\begin{align} \label{eq.26}
L= \frac{1}{2}\sigma^2 x^{2q(x)}\frac{d^2 }{d x^2}+\mu x^{p(x)}\frac{d}{d x}.
\end{align}
Next, we show that for the coefficients of $L$, the following conditions are satisfied:
\begin{itemize}
  \item [($\mathbf{a_1}$)] Since
\begin{equation}\label{eq.27}
\frac{1}{2}\sigma^2 x^{2q(x)}\geq \frac{1}{2}\sigma^2\,\min\{a^{2q^-}, a^{2q^+}\}:=\lambda,\,\, \lambda \in (0,\infty)
\end{equation}
$L$ is uniformly elliptic in $\mathcal{O}$.
  \item [($\mathbf{a_2}$)] $\sigma^2 x^{2q(x)}$ and $\mu x^{p(x)}$  are uniformly Lipschitz continuous on $\overline{\mathcal{O}}=[a,b]$.\\ We'll only provide a proof for $\vartheta(x):=x^{p(x)}$ since the same reasoning applies to $x^{2q(x)}$.
  Clearly $\vartheta$ is continuous on $\overline{\mathcal{O}}$. Further, using the assumptions ($\mathbf{h_1}$) and ($\mathbf{h_3}$), we obtain
\begin{align} \label{eq.28}
|\vartheta'(x)|\leq \max\{b^{p^-},b^{p^+}\}\left[\left(M_0+ C_0 a^{-(1+\alpha)}\right)\log(b)+\frac{p^+}{a} \right]
\end{align}
 which implies, by the Mean Value theorem, that $\vartheta$ is uniformly Lipschitz continuous on $\overline{\mathcal{O}}$.
\end{itemize}
Under these assumptions, it is well-known that Poisson’s equation \eqref{eq.1ab} admits a unique solution $u \in C^2(\mathcal{O})\cap C(\overline{\mathcal{O}})$.\\

\begin{theorem} \label{Thrm:5.1}
Suppose $u$ is the unique solution to Poisson’s equation \eqref{eq.1ab} in $\mathcal{O}$. Then $u$ is given by the Feynman–Kac formula
\begin{align} \label{eq.29}
u(x)=\mathbb{E}^{x}\left[\int_{0}^{\tau_{\mathcal{O}}} e^{-cs}\, f(X(s))ds\right],
\end{align}
where $\tau_{\mathcal{O}}=\inf\{t\geq0:  X(t) \notin \mathcal{O}\}$;  $\mathbb{E}^{x}[\cdot]$ means expectation conditioned on $X_0=x$; and $X(t)$ is the unique global solution of time-homogeneous nonlinear SDE with state-dependent variable exponent
\begin{equation}\label{eq.1abc}
\begin{cases}
\begin{array}{rll}
dX(t) &= \mu X^{p}(t)dt + \sigma X^{q}(t)dW(t),\,\, t\in [0,T], \\
X_{0}&=x \in \mathcal{O}.
\end{array}
\end{cases}
\end{equation}
\end{theorem}

\begin{proof}
Since $u \in C^{2}$, we can apply It\^{o}'s formula and obtain
\begin{align} \label{eq.30}
d(u(X(t))&=\left(\mu X^{p}(t)\frac{du(X(t))}{d x}+\frac{1}{2}\sigma^2 X^{2q}(t)\frac{d^2 u(X(t))}{d x^2} \right)dt\nonumber\\
&+\sigma X(t)^{q}\frac{d u(X(t))}{d x} dW(t),
\end{align}
and integrating over $[0,t]$
\begin{align} \label{eq.31}
u(X(t))&=u(X_{0})+\int_{0}^{t}\left(\mu X^{p}(s)\frac{du(X(s))}{d x}+\frac{1}{2}\sigma^2 X^{2q}(s)\frac{d^2 u(X(s))}{d x^2} \right)ds\nonumber\\
&+\int_{0}^{t}\sigma X^{q}(s)\frac{d u(X(s))}{d x} dW(s)\nonumber\\
&=u(X_{0})+\int_{0}^{t}Lu(X(s))ds+\int_{0}^{t}\sigma X^{q}(s)\frac{d u(X(s))}{d x} dW(s)
\end{align}
By the uniform ellipticity of $L$ and the boundedness of $\mathcal{O}$, $\mathbb{P}(\tau_{\mathcal{O}}=\infty)=0$; that is,  $\tau_{\mathcal{O}}<\infty$ $\mathbb{P}$-a.s. Let us define the approximated stopping time
\begin{align} \label{eq.31a}
S_n=\inf\left\{t\geq0: \text{dist}( X(t),\,\partial\mathcal{O})<\frac{1}{n}\right\}.
\end{align}
Note that with this definition, we have $S_n\leq S_{n+1}$ for all $n$, and thus
\begin{align} \label{eq.32}
S_n \leq \tau_{\mathcal{O}}\,\,\,  \text{ and }\,\,\, \lim_{n \to \infty}\,S_n=\tau_{\mathcal{O}}\,\,\, \mathbb{P}-a.s.
\end{align}
Then
\begin{align} \label{eq.32}
u(X_n)-u(X_{0})&=\int_{0}^{t\wedge S_n}Lu(X(s))ds+M_t.
\end{align}
where $X_n=X(t\wedge S_n)$ and
\begin{equation} \label{eq.33}
M_t=\int_{0}^{t\wedge S_n}\sigma X^{q}(s)\frac{d u(X(s))}{d x} dW(s).
\end{equation}
On the other hand, by Lemma \ref{Lem:2.1}, we know that $X^{q}(t) \in \mathcal{M_T}$. Thus, from \eqref{eq.1e} and \eqref{eq.20}, we can obtain
\begin{equation}\label{eq.34}
\mathbb{E}\left[\sup_{0\leq s \leq t\wedge S_n}X^{2q}(t)\right]\leq C \left(1+\mathbb{E}\left[\sup_{0 \leq s \leq t\wedge S_n} X^{2}(t)\right]\right) <\infty.
\end{equation}
Then, using the Fubini theorem and the fact that $u \in C^{2}$ in $\mathcal{O}$, we have
\begin{equation} \label{eq.35}
\int_0^{t\wedge S_n} \mathbb{E}\left[\left(\sigma X^{q}(s)\frac{d u(X(s))}{d x}\right)^2\right] ds< \infty.
\end{equation}
Thus, $M_t$ is a martingale. Next, applying the product formula to the process $e^{-c(t\wedge S_n)}\,u(X_n)$, and then using \eqref{eq.32} gives
\begin{align} \label{eq.36}
d\left(e^{-c(t\wedge S_n)}\,u(X_n)\right)&=e^{-c(t\wedge S_n)}\,du(X_n)+u(X_n)\,de^{-c(t\wedge S_n)}+d\langle e^{-c(t\wedge S_n)},\,u(X_n)\rangle_{t\wedge S_n}\nonumber\\
&=e^{-c(t\wedge S_n)}\,du(X_n)-ce^{-c(t\wedge S_n)}\mathbf{1}_{\{t<S_n\}} u(X_n)dt\nonumber\\
&=e^{-c(t\wedge S_n)}\mathbf{1}_{\{t<S_n\}}\left(Lu(X_n)-cu(X_n)\right)dt+ e^{-c(t\wedge S_n)}dM_t.
\end{align}
Integrating over $[0,t\wedge S_n)]$, and taking the expectation provides
\begin{align} \label{eq.37}
\mathbb{E}^{x}\left[e^{-c(t\wedge S_n)}\,u(X_n)\right]-u(x)&=\mathbb{E}^{x}\left[\int_{0}^{t\wedge S_n}e^{-cs}\left(Lu(X(s))-cu(X(s))\right)ds\right]\nonumber\\
&=-\mathbb{E}^{x}\left[\int_{0}^{t\wedge S_n}e^{-cs}f(X(s))ds\right].
\end{align}
Using the fact that
\begin{itemize}
  \item $\lim_{n \to \infty}\,S_n=\tau_{\mathcal{O}}$,
  \item $\lim_{t \to \infty}\,t\wedge \tau_{\mathcal{O}}=\tau_{\mathcal{O}}$,
\end{itemize}
and hence, letting first $n \to \infty$ and then $t \to \infty$ in \eqref{eq.37}, and considering the continuity of the sample paths and $u$, and the boundary condition,  the Dominated Convergence theorem gives
\begin{align} \label{eq.38}
u(x)&=\mathbb{E}^{x}\left[\int_{0}^{\tau_{\mathcal{O}}} e^{-cs}\, f(X(s))ds\right],
\end{align}
which concludes the proof.
\end{proof}

\section*{Acknowledgments}
\noindent This work was supported by Athabasca University Research Incentive Account [140111 RIA].
\section*{ORCID}
\noindent https://orcid.org/0000-0002-6001-627X

\end{document}